\documentclass[12pt]{amsart}
\usepackage{amssymb}
\usepackage{latexsym}

\newtheorem{theorem}{Theorem}[section]
\newtheorem{corollary}[theorem]{Corollary}
\newtheorem{lemma}[theorem]{Lemma}

\newtheorem{question}[theorem]{Question}

\newtheorem{historical background}[theorem]{Historical Background}

\newtheorem{remark}[theorem]{Remark}
\newtheorem{example}[theorem]{Example}
\newtheorem{notation}[theorem]{Notation}
\newtheorem{definition}[theorem]{Definition}
\newtheorem{proposition}[theorem]{Proposition}

\DeclareMathOperator\cl{\operatorname{cl}}

\def\cprime{$'$}

\def\sB {{\mathcal B}}

\def\sF {{\mathcal F}}
\def\sG {{\mathcal G}}

\def\sU {{\mathcal U}}
\def\sV {{\mathcal V}}
\def\sW {{\mathcal W}}

\def\cl {\mathrm{cl}}

\def\min {\mathrm{min}}
\def\sup {\mathrm{sup}}

\def\< {{\langle}}
\def\> {{\rangle}}

\begin{document}

%\vspace{0.5in}

\title[Generalizations of two inequalities of Hajnal and Juh\'asz]
{Generalizations of two cardinal inequalities of Hajnal and Juh\'asz}

%    Information for first author:
\author{Ivan S. Gotchev}
\address{Department of Mathematical Sciences, Central Connecticut Sta\-te 
University, 1615 Stanley Street, New Britain, CT 06050}
%    Current address (if needed):
%\curraddr{}
\email{gotchevi@ccsu.edu}
\thanks{The author expresses his gratitude to the Mathematics Department at the Universidad Aut\'onoma Metropolitana, Mexico City, Mexico, for their hospitality and support during his sabbatical visit of UAM in the spring semester of 2015.}

%    General info


\begin{abstract}
A non-empty subset $A$ of a topological space $X$ is called
\emph{finitely non-Hausdorff} if for every non-empty finite subset $F$ 
of $A$ and every family $\{U_x:x\in F\}$ of open neighborhoods $U_x$ of 
$x\in F$, $\cap\{U_x:x\in F\}\ne\emptyset$ and \emph{the non-Hausdorff number $nh(X)$ of $X$} is defined as follows: $nh(X):=1+\sup\{|A|:A\subset X$ is finitely 
non-Hausdorff$\}$. Clearly, if $X$ is a Hausdorff space then $nh(X)=2$.

We define the \emph{non-Urysohn number of $X$ with respect to the singletons}, $nu_s(X)$, as follows: $nu_s(X):=1+\sup\{\cl_\theta(\{x\}):x\in X\}$.

In 1967 Hajnal and Juh\'asz proved that if $X$ is a Hausdorff space then: (1) $|X|\le 2^{c(X)\chi(X)}$; and (2) $|X|\le 2^{2^{s(X)}}$; where $c(X)$ is the cellularity, $\chi(X)$ is the character and $s(X)$ is the spread of $X$.  

In this paper we generalize (1) by showing that if $X$ is a topological space then $|X|\le nh(X)^{c(X)\chi(X)}$. Immediate corollary of this result is that (1) holds true for every space $X$ for which $nh(X)\le 2^\omega$ (and even for spaces with $nh(X)\le 2^{c(X)\chi(X)}$). This gives an affirmative answer to a question posed by M. Bonanzinga in 2013. A simple example of a $T_1$, first countable, ccc-space $X$ is given such that $|X|>2^\omega$ and $|X|=nh(X)^\omega=nh(X)$. This example shows that the upper bound in our inequality is exact and that $nh(X)$ cannot be omitted (in particular, $nh(X)$ cannot always be replaced by $2$ even for $T_1$-spaces). 

In this paper we also generalize (2) by showing that if $X$ is a $T_1$-space then $|X|\le 2^{nu_s(X)\cdot 2^{s(X)}}$. It follows from our result that (2) is true for every $T_1$-space for which $nu_s(X)\le 2^{s(X)}$. A simple example shows that the presence of the cardinal function $nu_s(X)$ in our inequality is essential.
\end{abstract}

\maketitle

\section{Introduction}

Throughout this paper $\omega$ is (the cardinality of) the set of non-negative integers, $\xi$, $\eta$ and $\alpha$ are ordinals and $\kappa$ and $\tau$ are infinite cardinals. The cardinality of the set $X$ is denoted by $|X|$. By space we mean infinite topological space and for a subset $U$ of a space $X$ the closure of $U$ is denoted by $\overline{U}$.

Recall that a pairwise disjoint collection of non-empty open sets in a space $X$ is called a \emph{cellular family}. The \emph{cellularity} of $X$ is $c(X):=\sup\{|\sU|:\sU$ a cellular family in $X\}+\omega$. If $c(X)=\omega$ then it is called that $X$ satisfies the \emph{countable chain condition} (or \emph{ccc}) property. For $x\in X$ the character of $X$ at the point $X$ is $\chi(x,X):=\min\{|\mathcal{B}|:\mathcal{B}$ is a 
local base for $x\}$ and the character of $X$ is $\chi(X):=\sup\{\chi(x,X):x\in X\}$. 

In what follows, whenever $X$ is a space with $\chi(X)=\tau$ we shall assume that for each $x\in X$ a local base $\sB_x$ with $|\sB_x|\le\kappa$ has been fixed and if $A\subseteq X$ then by $\sU_A$ we shall denote the set of all families $\sU:=\{B_x:x\in A, B_x\in\sB_x\}$.

The following two definitions appeared in \cite{Got14}.

\begin{definition}\label{DIG1}
A non-empty subset $A$ of a topological space $X$  is called 
\emph{finitely non-Hausdorff} if, for every non-empty finite subset $F$ of $A$ 
and every $\mathcal{U}\in\mathcal{U}_F$, $\cap\mathcal{U}\ne\emptyset$.
The set $A$ is called a \emph{maximal finitely non-Hausdorff subset of 
$X$} if $A$ is a finitely non-Hausdorff subset of $X$ and if $B$ is a finitely 
non-Hausdorff subset of $X$ such that $A\subset B$ then $A=B$.
\end{definition}

\begin{definition}\label{DIG2}
Let $X$ be a topological space. \emph{The non-Hausdorff number 
$nh(X)$ of $X$} is defined as follows: $nh(X):=1+\sup\{|A|:A$ is a (maximal) 
finitely non-Hausdorff subset of $X\}$.
\end{definition}

M. Bonanzinga introduced in \cite{Bon13} (independently from \cite{Got14}) the notion of a \emph{Hausdorff number} of a topological space $X$, denoted $H(X)$, as follows: $H(X):=\min\{\tau:$ for every $A\subset X$ with $|A|\ge\tau$ there exist $\sU\in\sU_A$ such that $\cap\sU=\emptyset$ and she called \emph{$n$-Hausdorff} every space $X$ with $H(X)\le n$, where $n\in\omega$ and $n\ge 2$. It follows immediately from the definitions of $H(X)$ and $nh(X)$ that if $n\in\omega$ and $n\ge 2$ then $H(X)=n$ if and only if $nh(X)=n$. In the same paper Bonanzinga also introduced the notion of a \emph{weak Hausdorff number}, denoted $H^*(X)$, as follows: $H^*(X):=\min\{\tau:$ for every $A\subset X$ such that $|A|\ge \tau$ there exist $B\subset A$ with $|B|<\tau$ and $\sU\in \sU_B$ such that $\cap\sU=\emptyset\}$. She noted there that for every space $X$, $H^*(X)=H(X)$ or $H^*(X)=H(X)^+$ and constructed an example of a space $X$ such that $H^*(X)=H(X)=\omega$ (hence $H(X)\ne n$ for every $n<\omega$). It follows from the definitions that if $H^*(X)\le\omega$ then either $nh(X)=H(X)=n$ for some $n<\omega$ or $H(X)=\omega$ and $X$ is such that for every $n\in\omega$, $n\ge 2$ there is a finitely non-Hausdorff subset $A$ of $X$ with $|A|=n$ but there does not exist countably infinite finitely non-Hausdorff subset of $X$. Therefore if $H^*(X)\le\omega$ then $nh(X)\le\omega$. (Clearly, it is possible $nh(X)=\omega$ and $H^*(X)>\omega$).

In 1967, Hajnal and Juh\'asz proved that if $X$ is a  Hausdorff space then $|X|\le 2^{c(X)\chi(X)}$ (see \cite{HajJuh67}, \cite{Juh80} or \cite{Hod84}). Recently M. Bonanzinga showed that $|X|\le 2^{2^{c(X)\chi(X)}}$ whenever $X$ is a $3$-Hausdorff space (\cite[Corollary 54]{Bon13}) and she asked if the much more stronger inequality $|X|\le 2^{c(X)\chi(X)}$ holds true for every space $X$ with a finite Hausdorff number (\cite[Question 55]{Bon13}).

In this paper we prove that if $X$ is a topological space then $|X|\le nh(X)^{c(X)\chi(X)}$. Immediate corollary of this result is that the Hajnal--Juh\'asz inequality $|X|\le 2^{c(X)\chi(X)}$ holds true for every space $X$ for which $nh(X)\le 2^\omega$ (and even for spaces for which $nh(X)\le 2^{c(X)\chi(X)}$). This gives an affirmative answer of Bonanzinga's question. An Example of a $T_1$, first countable, ccc-space $X$ is given such that $|X|>2^\omega$ and $|X|=nh(X)^\omega=nh(X)$. This example shows that the upper bound in our inequality is exact and that $nh(X)$ cannot be omitted (in particular, $nh(X)$ cannot always be replaced by $2$ even for $T_1$-spaces).

\section{Some observations about $nh(X)$ and finitely non-Hausdorff subsets of $X$}

As it was noted in \cite{Got14}, it follows immediately from Definition \ref{DIG2} that $X$ is a Hausdorff space if and only if $nh(X)=2$ and $2< nh(X)\le 1+|X|$ whenever $X$ is a non-Hausdorff space. Also, if $X$ is a topological space and $A\subset X$, then $nh(A)\le nh(X)$, 
and if $X$ is an infinite set with topology generated by the open sets 
$\{X\setminus \{x\}:x\in X\}$, then $X$ is a maximal finitely non-Hausdorff 
set, and therefore $nh(X)=|X|$. 

The following three observations follow immediately from the definitions.

\begin{proposition}[\cite{Got14}]\label{P1}
In a Hausdorff space $X$ the only maximal finitely non-Hausdorff 
subsets of $X$ are the singletons.
\end{proposition}

\begin{proposition}[\cite{Got14}]\label{P2}
Every finitely non-Hausdorff subset of a topological space $X$ is contained in a 
maximal finitely non-Hausdorff subset of $X$. 
\end{proposition}

\begin{proposition}\label{P3}
Every subset of a finitely non-Hausdorff subset of a space $X$ is a finitely non-Hausdorff subset of $X$.
\end{proposition}

Having in mind Propositions \ref{P1} and \ref{P3} one can easily construct an example of a $T_1$-space $X$ and subsets $A$ and $B$ of $X$ such that $A\subset B$, $A$ and $B$ are finitely non-Hausdorff subsets of $X$ but $A$ is not finitely non-Hausdorff subset of $B$ (e.g. take $B=\alpha$ in Example \ref{E2} and let $A$ be any subset of $B$ which is not a singleton).

The following two observations appeared in \cite{Got14}. Since we are going to use them later, we give them here with their proofs. 

\begin{lemma}[\cite{Got14}]\label{L1}
Let $X$ be a space and $A$ be a finitely non-Hausdorff subset of 
$X$. Then $A\subset\cap\{\overline{\cap\mathcal{U}}:\mathcal{U}\in\mathcal{U}_F, \emptyset\neq F\subset A, |F|<\omega\}$.
\end{lemma}

\begin{proof}
Let $F$ be a non-empty subset of $A$, $\mathcal{U}_0\in\mathcal{U}_F$, and  
$G=\cap\mathcal{U}_0$. Suppose that there exists $a_0\in A$ such that 
$a_0\notin \overline{G}$. Then there is $W_{a_0}\in \mathcal{N}_{a_0}$ 
such that $W_{a_0}\cap G=\emptyset$. Let 
$V_{a_0}= W_{a_0}$ if $a_0\notin F$ and  
$V_{a_0}= U_{a_0} \cap W_{a_0}$, where 
$U_{a_0}\in \mathcal{U}_0$ and $U_{a_0}\in \mathcal{N}_{a_0}$,
if $a_0\in F$. Then the family $\mathcal{U}_1:=\{V_{a_0}\}\cup \{U_a:U_a\in\mathcal{U}_0,a\in F\setminus\{a_0\}\}$ has  
the property that $\cap\mathcal{U}_1=\emptyset$, a 
contradiction. Therefore $A\subset \overline{\cap\mathcal{U}}$ for 
every $\mathcal{U}\in\mathcal{U}_F$ and every non-empty subset $F$ of $A$ with 
$|F|<\omega$. Thus, 
$A\subset\cap\{\overline{\cap\mathcal{U}}:\mathcal{U}\in\mathcal{U}_F, \emptyset \neq F\subset A, |F|<\omega\}$.
\end{proof}

\begin{theorem}[\cite{Got14}]\label{T1}
Let $X$ be a space and $A$ be a maximal finitely 
non-Hausdorff subset of $X$. Then
$A=\cap\{\overline{\cap\mathcal{U}}:\mathcal{U}\in\mathcal{U}_F, \emptyset\neq F\subset A, |F|<\omega\}$.
\end{theorem}

\begin{proof}
Let $A$ be a maximal finitely non-Hausdorff subset of $X$. Then it follows from Lemma \ref{L1} that  $A\subset\cap\{\overline{\cap\mathcal{U}}:\mathcal{U}\in\mathcal{U}_F, \emptyset\neq F\subset A, |F|<\omega\}$. Suppose that there is 
$x_0\in \cap\{\overline{\cap\mathcal{U}}:\mathcal{U}\in\mathcal{U}_F, \emptyset\neq F\subset A, |F|<\omega\}\setminus A$. 
Then $U\cap(\cap\mathcal{U})\ne\emptyset$ for every 
$U\in \mathcal{N}_{x_0}$, every $\mathcal{U}\in\mathcal{U}_F$, and every non-empty finite subset $F$ of $A$. Thus, for the set $A_1:=A\cup\{x_0\}$, we have that if 
$F\subset A_1$ with $F\neq \emptyset$ and $|F|<\omega$ and
$\mathcal{U}\in\mathcal{U}_{F}$, then $\cap\mathcal{U}\ne\emptyset$.
Therefore, $A_1$ is a finitely non-Hausdorff subset of $X$ and 
$A\subsetneq A_1$, a contradiction with the maximality of $A$.
\end{proof}

\begin{corollary}\label{C1}
Every maximal finitely non-Hausdorff subset of a space $X$ is a closed set.
\end{corollary}

\begin{corollary}\label{C2}
Let $X$ be a space and $A$ be a finitely non-Hausdorff subset of 
$X$. Then $\overline{A}$ is a finitely non-Hausdorff subset of $X$.
\end{corollary}

\begin{proof}
It follows immediately from Proposition \ref{P2}, Corollary \ref{C1} and Proposition \ref{P3}.
\end{proof}

\begin{corollary}\label{C3}
Let $X$ be a space and $A$ be a finitely non-Hausdorff subset of 
$X$. If $x\in A$ then $A\subset\cap\{\overline{B}:B\in\sB_x\}$, hence $A\subset\cap_{x\in A}(\cap\{\overline{B}:B\in\sB_x\})$.
\end{corollary}

In relation to Corollary \ref{C3} we can say more.

\begin{lemma}\label{L2}
Let $X$ be a space and $x\in X$. Then $\cap\{\overline{B}:B\in \sB_x\}=\cup\{M:M$ is a (maximal) finitely non-Hausdorff subset of $X$ that contains $x\}$.
\end{lemma}

\begin{proof}
Let $y\in \cap\{\overline{B}:B\in \sB_x\}$ and $U$ be an open neighborhood of $y$. Then $U\cap B\ne\emptyset$ for every $B\in \sB_x$. Therefore the set $\{x,y\}$ is a finitely non-Hausdorff subset of $X$, hence it is contained in some maximal one. Therefore $\cap\{\overline{B}:B\in \sB_x\}\subseteq\cup\{M:M$ is a (maximal) finitely non-Hausdorff subset of $X$ that contains $x\}$.

Now let $y\in\cup\{M:M$ is a (maximal) finitely non-Hausdorff subset of $X$ that contains $x\}$. Then there exists a (maximal) finitely non-Hausdorff subset $M_y$ of $X$ such that $y\in M_y$. Then the set $\{x,y\}\subset M_y$ is a finitely non-Hausdorff subset of $X$ (Proposition \ref{P3}). Thus if $B\in\sB_x$ and $U$ is an arbitrary open neighborhood of $y$ we have $B\cap U\ne\emptyset$. Hence $y\in\overline{B}$ and therefore $y\in\cap\{\overline{B}:B\in \sB_x\}$.
\end{proof}

\begin{corollary}\label{C4}
Let $X$ be a space and $A$ be a finitely non-Hausdorff subset of 
$X$. Then $\cap_{x\in A}(\cap\{\overline{B}:B\in\sB_x\})=\cap_{x\in A}(\cup\{M:M$ is a (maximal) finitely non-Hausdorff subset of $X$ that contains $x\})$.
\end{corollary}

The following example shows that the intersection in Corollary \ref{C4} could be different from $A$ even when $A$ is a   maximal non-Hausdorff subset of $X$.

\begin{example}\label{E1}
There exists a first countable $T_1$-space $X$ and a maximal finitely 
non-Hausdorff subset $A$ of $X$ such that 
$$A\subsetneq\cap_{x\in A}(\cap\{\overline{B}:B\in\sB_x\}).$$  
\end{example}

\begin{proof}
Let $\{b\}$, $A:=\{a_i:i\in\omega\}$, $S:=\{n:n\in\omega\}$, and  $N^2:=\{(i,n):i,n\in\omega\}$ be pairwise disjoint sets and $X:=A\cup \{b\}\cup S\cup N^2$. We define topology on $X$ as follows: all points in $S\cup N^2$ are isolated (hence each one is an open and closed set); the points in $S$ form a convergent sequence that converges to the point $a_i$ for every $i\in\omega$; for each $i\in\omega$ the points $\{(i,n):n\in\omega\}$ form a convergent sequence that approaches to the points $a_i$ and $b$. In order $X$ to be first countable we also require the set $\{\{b\}\cup\{(i,n):i\in\omega\}:n\ge k\}:k\in\omega\}$ to form an open basis for the topology at $b$. Then $A$ and each of the sets $\{a_i, b\}$, $i\in\omega$, are maximal finitely non-Hausdorff subsets of $X$, $A\cup\{b\}$ is not a finitely non-Hausdorff subset of $X$ and $A\subsetneq A\cup\{b\}=\cap_{x\in A}(\cap\{\overline{B}:B\in\sB_x\})$.
\end{proof}

We recall that the \emph{$\theta$-closure} of a set $A$ in a space $X$ is the set 
$\cl_\theta(A) := \{x\in X :$ for every $B\in \mathcal{B}_x, \overline{B}\cap A \ne \emptyset\}.$

\begin{proposition}\label{P4}
Let $X$ be a space and $x\in X$. Then $\cl_\theta(\{x\})=\cap\{\overline{B}:B\in \sB_x\}$.
\end{proposition}

\begin{corollary}\label{C5}
Let $X$ be a space and $x\in X$. Then $\cl_\theta(\{x\})=\cup\{M:M$ is a (maximal) finitely non-Hausdorff subset of $X$ that contains $x\}$.
\end{corollary}

For convenience we introduce the following notation to be used in the proof of our main result.

\begin{notation}
Let $X$ be a space and $A\subseteq X$. Then 
\begin{equation*}
\sF_A:=\{F:F\subset A, F\text{ is a finite, finitely non-Hausdorff subset of }X\}.
\end{equation*}
\end{notation}

Using this notation we can restate Corollary \ref{C5} as follows:

\begin{corollary}\label{C6}
Let $X$ be a space and $x\in X$. Then $\cl_\theta(\{x\})=\cup\{F:F\in \sF_X, x\in F\}$.
\end{corollary}

\begin{corollary}\label{C6}
Let $X$ be a space and $x\in X$. The union of all (maximal) finitely non-Hausdorff subsets of $X$ that contain $x$ is a closed set in $X$.
\end{corollary}

We recall that a non-empty subset $A$ of a topological space $X$  is called
\emph{finitely non-Urysohn} (see \cite{Got15}) if for every non-empty finite subset $F$ of $A$ and 
every family $\{U_x:x\in F\}$ of open neighborhoods $U_x$ of $x\in F$, 
$\cap\{\overline{U_x}:x\in F\}\ne\emptyset$ and \emph{the non-Urysohn number of $X$} 
is defined as follows: $nu(X):=1+\sup\{|A|:A$ is a finitely non-Urysohn subset of 
$X\}$.

\begin{corollary}\label{C7}
Let $X$ be a space and $x\in X$. Then $\cl_\theta(\{x\})$ is a finitely non-Urysohn subset of $X$ that contains $x$.
\end{corollary}

\begin{corollary}\label{C8}
Let $X$ be a space and $x\in X$. Then $nu(X)\ge |\cup\{M:M$ is a maximal finitely non-Hausdorff subset of $X$ that contains $x\}|$.
\end{corollary}

\begin{corollary}\label{C8}
If $X$ is a space then $nu(X)\ge nh(X)$.
\end{corollary}

We finish this section with one more observation.

\begin{lemma}\label{L3}
Let $X$ be a space and $x\in X$ be a point such that $\overline{U}=X$ whenever $U\subset X$ is an open neighborhood of $x$. If $M$ is a maximal finitely non-Hausdorff subset of $X$ then $x\in M$.
\end{lemma}

\begin{proof}
Let $M$ be a maximal finitely non-Hausdorff subset of $X$. Suppose that $x\notin M$. Then there exist a finite set $F\subset M$, $\sU\in\sU_F$ and an open neighborhood $U$ of $x$ such that $(\cap\sU)\cap U=\emptyset$. Since $M$ is a finitely non-Hausdorff subset of $X$, $\cap\sU\ne\emptyset$. Let $y\in\cap\sU$. Then $y\notin\overline{U}$ - contradiction. \end{proof}

\section{More cardinal inequalities involving the non-Hausdorff number}

The following theorem generalizes Hajnal--Juh\'asz inequality that if $X$ is a Hausdorff space then $|X|\le 2^{c(X)\chi(X)}$. 

\begin{theorem}\label{TIG1}
Let $X$ be a space. Then $|X|\le nh(X)^{c(X)\chi(X)}$.
\end{theorem}
\begin{proof}
Let $c(X)\chi(X)=\kappa$, $nh(X)=\tau$ and for each $x\in X$ let $\sB_x$ be a local base for $x$ in $X$ with $|\sB_x|\le\kappa$. Let also $x_0$ be an arbitrary point in $X$.  We construct a sequence $\{G_\eta:\eta<\kappa^+\}$ of subsets of $X$ such that
\begin{itemize}
\item[(1)] $G_0=\{x_0\}$;
\item[(2)] $\cup_{\xi<\eta}G_\xi\subset G_\eta$ and $|G_\eta|\le \tau^\kappa$ for every $0<\eta<\kappa^+$;
\item[(3)] if $\eta$ is a limit ordinal then $G_\eta=\cup_{\xi<\eta}G_\xi$;
\item[(4)] if $x\in G_\eta$ then there exists a maximal finitely non-Hausdorff subset $M_x$ of $X$ such that $M_x\subset G_{\eta+1}$;
\item[(5)] if $\{W_\xi:\xi<\kappa\}$ is a collection of $\le\kappa$ open sets in $X$ such that 
$W_\xi=\cup_{\alpha<\kappa}\sG_\alpha$, where each $\sG_\alpha=\cap\sU_{\xi,\alpha}$ for some $\sU_{\xi,\alpha}\in\sU_{F_{\xi,\alpha}}$ with $F_{\xi,\alpha}\in\sF_{G_\eta}$, and  $\cup_{\xi<\kappa}\overline{W}_\xi\ne X$, then $G_{\eta+1}\setminus(\cup_{\xi<\kappa}\overline{W}_\xi)\ne\emptyset$.
\end{itemize}
Let $G=\cup_{\xi<\kappa^+}G_\eta$. If $G=X$ then the proof is complete. Suppose there is $y\in X\setminus G$. Let $\sB_y=\{B_\xi:\xi<\kappa\}$ be a local base at $y$. For each $\xi<\kappa$ let $\sW_\xi=\{\cap\sU:\sU\in\sU_F, F\in\sF_G, (\cap\sU)\cap B_\xi=\emptyset\}$. Then clearly $y\notin\cup_{\xi<\kappa}\overline{\cup\sW_\xi}$.
\vskip12pt
\textbf{Claim 1:} For each $x\in G$, there exist $F\in\sF_G$ with $x\in F$, $\sU\in\sU_F$ and $\xi<\kappa$ such that $(\cap\sU)\cap B_\xi=\emptyset$. 

\begin{proof}[\textbf{Proof:}] 
Let $x\in G$. Then there is $\eta<\kappa^+$ such that $x\in G_\eta$. It follows from (4) that there exists a maximal finitely non-Hausdorff subset $M_x$ of $X$ such that $x\in M_x\subset G_{\eta+1}$. Since $y\notin G$ we have $y\notin M_x$. Thus, there exists a finite set $F_x\subset M_x$, $\sU'\in\sU_{F_x}$ and $\xi<\kappa$ such that $(\cap\sU')\cap B_\xi=\emptyset$. Therefore for the finitely non-Hausdorff subset $F:=F_x\cup\{x\}\subset M_x$ of $\sF_G$ there exists $\sU\in\sU_{F}$ such that $(\cap\sU)\cap B_\xi=\emptyset$.
\end{proof}

\textbf{Claim 2:} $G\subseteq\cup_{\xi<\kappa}\overline{\cup\sW_\xi}$. 

\begin{proof}[\textbf{Proof:}] 
Let $x\in G$. It follows from Claim 1 that there exists $F\in\sF_G$ with $x\in F$, $\sU\in\sU_F$ and $\xi<\kappa$ such that $(\cap\sU)\cap B_\xi=\emptyset$. Then it follows from Lemma \ref{L1} that $F\subset\overline{\cap\sU}$, hence $x\in\overline{\cup\sW_\xi}$ and therefore $x\in\cup_{\xi<\kappa}\overline{\cup\sW_\xi}$.
\end{proof}

Since $c(X)\le\kappa$, there exists $\sG_\xi\subseteq\sW_\xi$ with $|\sG_\xi|\le\kappa$ such that $\cup\sW_\xi\subseteq\overline{\cup\sG_\xi}$. Let $W_\xi=\cup\sG_\xi$. Then $G\subseteq\cup_{\xi<\kappa}\overline{W}_\xi$ and $y\notin \cup_{\xi<\kappa}\overline{W}_\xi$. Let $\sF:=\{F:$ there is $\sU\in \sU_F$ and $\xi<\kappa$ such that $\cap\sU\in\sG_\xi\}$. Since $|\sF|\le\kappa$, we have $|\cup\sF|\le\kappa$. Clearly $\cup\sF\subset G$. Thus, we can find $\eta<\kappa^+$ such that $\cup\sF\subset \sF_{G_\eta}$. Then it follows from (5) that $G_{\eta+1}\setminus(\cup_{\xi<\kappa}\overline{W}_\xi)\ne\emptyset$. This contradicts $G\subseteq\cup_{\xi<\kappa}\overline{W}_\xi$.
\end{proof}

\begin{corollary}\label{C9}
Let $X$ be a space with $nh(X)\le 2^{c(X)\chi(X)}$. Then $|X|\le 2^{c(X)\chi(X)}$.
\end{corollary}

Corollary \ref{C9} answers in the affirmative the following question of Bonanzinga (see \cite[Question 55]{Bon13}): Is $|X|\le 2^{c(X)\chi(X)}$ true for every $X$ such that $H(X)$ is finite? It also greatly improves her Corollary 54 that states that for every 3-Hausdorff space $X$, $|X|\le 2^{2^{c(X)\chi(X)}}$.

The following example shows that Hajnal--Juh\'asz inequality is not always true for $T_1$-spaces and that the cardinal number $nh(X)$ in the inequality in Theorem \ref{TIG1} cannot be replaced by $2$. For a different example of a $T_1$-space for which Hajnal--Juh\'asz inequality is not true see \cite[Example 13]{Bon13}.

\begin{example}[see {\cite[Example 2.1]{Got14}}]\label{E2}
Let $\mathbb{N}$ denote the set of all positive integers and  
$\mathbb{R}$ be the set of all real numbers. Let 
$S:=\{1/n: n\in \mathbb{N}\}$ and $M:=S\cup\{0\}$ be the subspace of 
$\mathbb{R}$ with the inherited topology. Then in $M$ all points except $0$ 
are isolated and $\lim_{n\rightarrow\infty}1/n=0$. Let $\alpha$ be an infinite 
initial ordinal. We duplicate $\alpha$ many times the point $0\in M$; i.e. we 
replace in $M$ the point $0$ with $\alpha$ many distinct points and obtain 
the set $X:=S\cup\alpha$ with topology such that, for each 
$\beta<\alpha$, we have $\beta\in\lim_{n\rightarrow\infty}1/n$ and the 
subspaces $S$ and $\alpha$ with the inherited topology from $X$ are 
discrete. Then $X$ is $T_1$ (but not Hausdorff) ccc-space, $\chi(X)=\omega$, and $nh(X)=\alpha$. Therefore if $\alpha>2^\omega$ is a cardinal for which $\alpha^\omega=\alpha$ then $|X|=nh(X)^{c(X)\chi(X)}=\alpha^\omega=\alpha>2^\omega$.
\end{example}

Another well-known inequality of Hajnal and Juh\'asz is contained in the next theorem.

\begin{theorem}[Hajnal-Juh\'asz]\label{THJ2}
If $X$ is a Hausdorff space then $|X|\le 2^{2^{s(X)}}$.
\end{theorem}

In Theorem \ref{TIG2} we generalize Theorem \ref{THJ2} for the class of $T_1$-spaces.
In the proof of Theorem \ref{TIG2} we will need the following three results (see \cite{Juh80} or \cite{Hod84}):

\begin{lemma}[\v{S}apirovski{\u\i}]\label{L4}
Let $\sU$ be an open cover of a space $X$ with $s(X)\le\kappa$. Then there is a subset $A$ of $X$ with $|A|\le \kappa$ and a subcollection $\sW$ of $\sU$ with $|\sW|\le\kappa$ such that $X=\overline{A}\cup(\cup\sW)$ .
\end{lemma}

\begin{lemma}\label{L5}
If $X$ is a Hausdorff space then $\psi(X)\le 2^{s(X)}$.
\end{lemma}

\begin{theorem}[Hajnal-Juh\'asz] \label{THJ1}
If $X$ is a $T_1$-space then $|X|\le 2^{s(X)\psi(X)}$.
\end{theorem}

In order to extend Lemma \ref{L5} to the class of all $T_1$-spaces we need to introduce a new cardinal function.

\begin{definition}
Let $X$ be a space. We define the \emph{non-Urysohn number of $X$ with respect to the singletons}, $nu_s(X)$, as follows: $nu_s(X):=1+\sup\{\cl_\theta(\{x\}):x\in X\}$.
\end{definition}

Clearly if $X$ is a Hausdorff space then  $nu_s(X)=2$ and for every space $nu_s(X)\le nu(X)$ and $nu_s(X)\ge nh(X)$ (see Corollary \ref{C5} and Corollary \ref{C7}).

\begin{lemma}\label{L6}
If $X$ is a $T_1$-space then $\psi(X)\le nu_s(X)\cdot 2^{s(X)}$.
\end{lemma}

\begin{proof}
Let $s(X)=\kappa$, $nu_s(X)=\tau$ and $x\in X$. Using the fact that $X$ is a $T_1$-space, for each $z\in \cl_\theta(\{x\})$ we can choose an open neighborhood $U_z$ of $x$ that does not contain $z$. Also, for each $y\notin \cl_\theta(\{x\})$ we can choose an open set $U_y$ such that $x\notin \overline{U_y}$. Then 
$\sU:=\{U_y:y\notin\cl_\theta(\{x\})\}$ is an open cover of $X\setminus\cl_\theta(\{x\})$. Therefore, according to Lemma \ref{L4}, there exist subsets $A$ and $B$ of $X\setminus\cl_\theta(\{x\})$ such that $|A|\le\kappa$, $|B|\le\kappa$ and $X\setminus\cl_\theta(\{x\})\subseteq\overline{A}\cup(\cup_{y\in B}U_y)$. Let $\sV:=\{U_z:z\in\cl_\theta(\{x\})\}$, $\sV_A:=\{X\setminus\overline{U_y\cap A}:y\in\overline{A}\setminus\{x\}\}$ and $\sV_B:=\{X\setminus\overline{U_y}:y\in B\}$. Then $\sV\cup\sV_A\cup\sV_B$ is a pseudobase for $X$ with cardinality $\le\tau+2^\kappa + \kappa\le \tau\cdot 2^\kappa$.
\end{proof}

\begin{corollary}
If $X$ is a $T_1$-space then $\psi(X)\le nu(X)\cdot 2^{s(X)}$.
\end{corollary}

\begin{remark}
Let $\kappa>2^\omega$ and $X$ be a space with cardinality $\kappa$, equipped with the cofinite topology. Then  $s(X)=\omega$ and $\psi(X)=\kappa$. Hence $\psi(X)=\kappa>2^\omega=2^{\psi(X)}$. Also, for every $x\in X$ we have $\cl_\theta(\{x\})=X$. Thus $nu_s(X)=\kappa$. Therefore in the inequality in Lemma \ref{L6} the cardinal function $nu_s(X)$ cannot be replaced by 2, but we do not know if $nu_s(X)$ cannot be replaced by $nh(X)$. (Note that in our example $nh(X)=\kappa$, as well.)
\end{remark}

Now using Lemma \ref{L6} and Theorem \ref{THJ1} we generalize Theorem \ref{THJ2} as follows:

\begin{theorem}\label{TIG2}
If $X$ is a $T_1$-space then $|X|\le 2^{nu_s(X)\cdot 2^{s(X)}}$.
\end{theorem}

\begin{proof}
$|X|\le 2^{s(X)\psi(X)}\le 2^{s(X)\cdot nu_s(X)\cdot 2^{s(X)}} = 2^{nu_s(X)\cdot 2^{s(X)}}$.
\end{proof}

\begin{corollary}
If $X$ is a $T_1$-space such that $nu_s(X)\le 2^{s(X)}$ then $|X|\le 2^{2^{s(X)}}$.
\end{corollary}

\begin{corollary}
If $X$ is a $T_1$-space then $|X|\le 2^{nu(X)\cdot 2^{s(X)}}$.
\end{corollary}

\begin{question}
Is it true that if $X$ is a $T_1$-space then $|X|\le 2^{nh(X)\cdot 2^{s(X)}}$?
\end{question}

\end{document}